\documentclass[10pt, reqno]{amsart}
\usepackage{amsmath}
\usepackage{amssymb}
\usepackage{enumerate}
\usepackage{amsbsy}
\usepackage{amsfonts}
\usepackage{color}
\usepackage{dsfont,amsxtra,latexsym,amscd}
\usepackage{mathabx}
\usepackage{bbm}

\numberwithin{equation}{section}
\numberwithin{figure}{section}

\newtheorem{theorem}{Theorem}[section]

\newtheorem{lemma}[theorem]{Lemma}
\newtheorem{proposition}[theorem]{Proposition}

\theoremstyle{definition}

\newtheorem{remark}[theorem]{Remark}

\makeatletter
\renewenvironment{proof}[1][\proofname]{%
   \par\pushQED{\qed}\normalfont%
   \topsep6\p@\@plus6\p@\relax
   \trivlist\item[\hskip\labelsep\bfseries#1\@addpunct{.}]%
   \ignorespaces
}{%
   \popQED\endtrivlist\@endpefalse
}
\makeatother

\newcommand{\la}{\langle}
\newcommand{\ra}{\rangle}
\newcommand{\ls}{\lesssim}

\newcommand{\pa}{\partial}

\newcommand{\ep}{\epsilon}

\newcommand{\R}{\mathbb{R}}

\begin{document}

\title[KdV limit for the Boussinesq equation]{On the Korteweg-de Vries limit for the Boussinesq equation}

\linespread{1.2}

\author[Y. Hong]{Younghun Hong}
\address{Department of Mathematics, Chung-Ang University, Seoul 06974, Korea}
\email{yhhong@cau.ac.kr}
 
 \author[C. Yang]{Changhun Yang}
\address{Department of Mathematics, Chungbuk National University, Cheongju-si 28644, Chungcheongbuk-do, Korea}
\email{chyang@chungbuk.ac.kr}

\begin{abstract}
The Korteweg-de Vries (KdV) equation is known as a universal equation describing various long waves in dispersive systems. In this article, we prove that in a certain scaling regime, a large class of rough solutions to the Boussinesq equation are approximated by the sums of two counter-propagating waves solving the KdV equations. It extends the earlier result \cite{Schneider1998} to slightly more regular than $L^2$-solutions. Our proof is based on robust Fourier analysis methods developed for the low regularity theory of nonlinear dispersive equations. 
\end{abstract}

\maketitle

\section{Introduction}

The Korteweg-de Vries (KdV) equation has been introduced at first by Boussineq \cite{B1872} and Korteweg and de Vries \cite{KdV1895} as a model equation describing propagation of shallow water surfaces along a channel. However, it has played much more significant roles in vast areas in pure and applied mathematics, including the discovery of integrable systems and the development of the inverse scattering method. As for analysis of PDEs, the KdV equation has drawn huge attention in spite of its simplicity, because the derivative nonlinearity makes proving well-posedness challenging. There is a long history of the low regularity well-posedness theory for the KdV equation, which includes the celebrated work of Bourgain \cite{B-1993KdV}, Kenig-Ponce-Vega \cite{KPV-1991,KPV-1993DUKE, KPV-1993CPAM, KPV1996} and Colliander-Keel-Staffilani-Takaoka-Tao \cite{CKSTT-2003}, where general Fourier analysis methods have been developed for dispersive equations. Relatively recently, the optimal well-posedness is established in Kappeler and Topalov \cite{KT2006} and Killip and Visan \cite{KV-2019}, invoking the complete integrability of the equation.

On the other hand, in mathematical modeling, the KdV equation is thought to be a \textit{universal} modulation equation for slowly varying long waves. It can be derived rigorously from various dispersive equations such as the water wave problem \cite{Cra1985,SW2000,SW2022,Du2012}, the Vlasov(-type) equations in plasma physics \cite{HanKwan2013, GP2014}, and the Fermi-Pasta-Ulam (FPU) system \cite{ZK-1965,SW-2000,HKY2021}. For more details, we refer to Schneider-Uecker \cite[Chapter 12]{SU-2017}.

In this article, we are particularly interested in rigorous derivation of the KdV equation from the ``good" Boussinesq equation
\begin{equation}\label{Boussinesq}
   \pa_t^2 u - \pa_x^2u +\pa_x^4 u +\pa_x^2( u^2)=0,
  \end{equation}
where $u(t,x):\R\times\R\rightarrow \R$. The equation \eqref{Boussinesq} is one of simplest models for dispersive fluids, and it can be derived from water wave problem modeling a two dimensional irrotational flow of an inviscid incompressible fluid in an infinitely long canal of fixed finite depth \cite{B1872}.  It is proved by Linares \cite{Linares1993} that the Boussinesq equation is globally well-posed for small data in the energy space, and its regular solutions preserves the energy defined by  
\begin{align}\label{energy}
   E(u(t))=\frac{1}{2}\|\partial_x^{-1}\pa_t u(t)\|_{L^2(\R)}^2+\frac{1}{2}\|\partial_x u(t)\|_{L^2(\R)}^2+\frac{1}{2}\| u(t)\|_{L^2(\R)}^2-\frac{1}{3}\|u(t)\|_{L^3(\R)}^3.
\end{align}
Moreover, by the Fourier analysis methods, the local well-posedness in a low regularity space is established by Farah \cite{Fa2009} and Kishimoto \cite{No2013}.

It is proved in Schneider \cite{Schneider1998} that if slowly modulating solutions to the Boussinesq equation are sufficiently regular (see Remark \ref{main theory remarks}), then they can be approximated by counter-propagating KdV flows. More precisely, inserting the long wave ansatz
\begin{equation}\label{KdV ansatz}
   u(t,x)=\ep^2 w_+(\ep(x-t),\ep^3t) + \ep^2 w_-(\ep(x+t),\ep^3t) 
\end{equation}
for a small parameter $0<\ep\ll1$ into \eqref{Boussinesq} and equating the coefficients at $\ep^6$-order terms to zero yields the system of two decoupled KdV equations
\begin{align}\label{KdV}
   2\pa_t w_\pm \mp \pa_x ^3 w_\pm \mp \pa_x (w_\pm ^2)=0, 
\end{align}
where $w_\pm(t,x):\R\times \R\rightarrow\R$. The approximation is valid on the time scale of order $\mathcal O(1/\ep^3)$.

Our main result extends this earlier result to a larger class of low regularity solutions.
\begin{theorem}[Korteweg-de Vries limit for the Boussinesq equation]\label{main theorem}
Let $\ep\in(0,1)$ and $0<s\le 5$. Suppose that
$$\sup_{\epsilon\in(0,1]}\|u_{\epsilon,0}^{\pm}\|_{H^s(\R)} \le R,$$
and let $u_\epsilon(t)\in C_t(\mathbb{R};H^s(\R))$ be the solution to the Boussinesq equation \eqref{Boussinesq} with initial data
$$\big(u_\epsilon(0,x),\partial_tu_{\epsilon}(0,x)\big)=\epsilon^2\Big((u_{\epsilon,0}^++u_{\epsilon,0}^-)(\epsilon x), \partial_x\sqrt{1-\partial_x^2}\big((u_{\epsilon,0}^--u_{\epsilon,0}^+)(\epsilon x)\big)\Big).$$
Then, there exists $T(R)>0$ such that 
\begin{equation}\label{eq: convergence}
\sup_{{|t|\leq\frac{T}{\epsilon^3}}}\big\| u_\ep(t,x)- \epsilon^2w_\epsilon^+\big(\epsilon^3t, \epsilon(x-t)\big) -  \epsilon^2w_\epsilon^-\big(\epsilon^3t, \epsilon(x+t)\big) \big\|_{L^2(\R)} \lesssim_{R}\ep^{\frac{3}{2}+\min\{\frac25s,\frac12\}},
\end{equation}
where $w_\epsilon^{\pm}\in C_t(\mathbb{R};H^s(\R))$ is the solution to the KdV equation \eqref{KdV} with the same initial data $u_{\epsilon,0}^{\pm}$.
\end{theorem}

\begin{remark}[{Reduced regularity}]\label{main theory remarks}
$(i)$ Compared to \cite[Theorem 1]{Schneider1998} and \cite[Theorem 12.1.1 and 12.4.1]{SU-2017}, the regularity requirement for the KdV limit  \eqref{eq: convergence} is reduced down to $s>0$ from $s>\frac{9}{2}$, as well as a weighted norm bound $\|(1+|x|^2)u_{\epsilon,0}^{\pm}\|_{L^2}$ on initial data is removed for the counter-propagating wave case.\\
$(ii)$ Lowering regularity and removing a weighed norm bound condition for the KdV limit might not be just a purely technical question, but it can also be used to extend the interval of validity. Indeed, in Theorem \ref{main theorem}, the time interval 
$[-\frac{T}{\epsilon^3},\frac{T}{\epsilon^3}]$ is given in terms of ``sufficiently small" $T>0$. However, if one could prove such a convergence for energy class solutions, there would be a hope to invoke the conservation law to show that $\|...(t)\|_{L^2}$ in \eqref{eq: convergence} is bounded by $\sim e^{c(1+|t|)^m}$ or $\sim (1+|t|)^m$ for ``any" $t\in\mathbb{R}$. Then, the interval of validity would be extended logarithmically or polynomially in $\frac{1}{\epsilon}$. \\
$(iii)$ For fixed $\epsilon>0$, the conservation law of the Boussinesq equation \eqref{energy} controls the $H^1$-norms of solutions. Nevertheless, by scaling, the conservation law loses its control on $\|\partial_x u_\epsilon(t)\|_{L^2}^2$ in the limit $\epsilon\to0$. It merely provides a bound for the $L^2$-norm (see \eqref{rescaled energy}). For this reason, it is desirable to obtain \eqref{eq: convergence} for $L^2$(or $L^2$ type)-solutions, which will be postponed in our future work.
\end{remark}

\begin{remark}[Rate of convergence]
$(i)$ In \eqref{eq: convergence}, the $O(\epsilon^{\frac{3}{2}})$-factor is simply from the long wave scaling \eqref{KdV ansatz}. The additional $O(\epsilon^{\frac{2}{5}s})$-factor is the essential rate of convergence.\\
$(ii)$ When $0\leq s\leq \frac{5}{4}$, the additional $O(\epsilon^{\frac{2}{5}s})$-factor in \eqref{eq: convergence} is optimal in the sense that the difference between the corresponding two flows is bounded by $O(\epsilon^{\frac{2}{5}s})$ (see Lemma \ref{Lem:difference linear sol in Xsb}).\\
$(iii)$ Since the $O(\epsilon^{\frac{7}{2}})$-bound is obtained for localized regular solutions \cite{Schneider1998}, one may expect that the bound \eqref{eq: convergence} can be improved to $O(\ep^{\frac{3}{2}+\min\{\frac25s,2\}})$. However, we currently do not know how to prove it. Indeed, such a loss is from the bilinear estimates (Lemma \ref{Lem:BL2} and \ref{Lem:BL3}) which measure interactions between the counter-propagating waves. There is also a possibility that putting a weighted norm bound (see Remark \ref{main theory remarks} $(i)$) might be necessary for a better rate of convergence, because localization reduces interactions of two waves.
\end{remark}

In the previous work \cite{Schneider1998}, the KdV limit is justified by multi-scaling analysis, involving energy estimates for the residual $\textup{Res}(u) = -\pa_t^2 u + \pa_x^2u -\pa_x^4 u -\pa_x^2( u^2)$, where $u$ is the ansatz \eqref{KdV ansatz}, but this approach requires high Sobolev norm bounds on $w_{\pm}$. On the other hand, the proof of our main result is based on the well-established low regularity well-posedness theory, involving Fourier analysis methods, for nonlinear dispersive equations. A key observation is that reformulating the problem as a system of integral equations (see \eqref{Coupled Boussinesq}), or the Duhamel representation, enables us to capture dispersive effects properly. Indeed, the linear evolutions in the reformulated equation \eqref{Coupled Boussinesq} are given by $e^{\mp \frac{t}{\ep^2}\pa_x(\sqrt{1-\epsilon\Delta}-1)}$. Therefore, developing the linear and bilinear estimates associated with these linear flows, one can deduce uniform bounds for low regularity nonlinear solutions.

By similar approaches, regularity requirements have been reduced for limit problems in recent work of the authors and their collaborators; \cite{HY-2019SIAM,HKNY2021,HKY2023} for continuum limit for discrete nonlinear Schr\"odinger equations, and \cite{HKY2021} for the KdV limit for the Fermi-Pasta-Ulam (FPU) system. We note that in some simpler cases \cite{HY-2019SIAM,HKNY2021,HKY2023}, the interval of validity is extended arbitrarily by combining the energy conservation law.

For Theorem \ref{main theorem}, we mainly follow the strategy in Hong-Kwak-Yang \cite{HKY2021}, where $H^s$-solutions to the KdV equation are derived from the FPU system, provided that $s>\frac{3}{4}$. In \cite{HKY2021}, the obstacle to go below $\frac{3}{4}$ is from that the Fourier restriction norm cannot be directly employed to measure the difference of two different equations (see \cite[Appendix A]{HKY2021}). The same difficulty also arises for the limit from the Boussinesq model. In this paper, a new idea is introduced to overcome this problem employing an auxiliary equation, namely, the frequency localized decoupled Boussinesq equation (see \eqref{LFB} below). It allows us to reduce the regularity requirement up to $s>0$, and we expect that the same idea can be applied to other models. More detailed outline of the proof will be given in the next section.

\subsection{Notations}
Throughout this article, $\epsilon\in (0,1]$ is a small parameter which will be sent to zero. A very important remark is that we always denote
$$A\lesssim B\quad\textup{(resp., }A\gtrsim B,\ A\sim B)$$
if there exists $c>0$, ``independent of $\epsilon\in (0,1]$", such that $A\leq c B$ (resp., $A\geq c B$, $\frac{1}{c}A\leq B\leq cA$). The standard Japanese bracket is expressed by $\langle x\rangle:=\sqrt{1+x^2}$. Let $\eta$ be a fixed smooth bump function $\eta\in C_c^\infty(\R)$ such that 
\begin{equation}\label{bump function}
0\leq \eta\leq 1,\quad \eta\equiv 1\textup{ on }[-1,1],\quad\textup{supp}\eta\subset[-2,2],
\end{equation}
and we define the ``sharp'' frequency truncation operator $P_N$ by
\begin{equation}\label{frequency truncation operator}
\widehat{ P_{\le N}u} (\xi) = \mathbf{1}_{\{|\xi|\le N\}}\hat{u}(\xi)\quad\textup{with}\quad N=\frac12\ep^{-\frac25}.
\end{equation}

\subsection{Acknowledgement}
The first author was supported by the National Research Foundation of Korea(NRF) grant funded by the Korea government(MSIT) (No. RS-2023-00208824 and No. RS-2023-00219980). The second author was supported by the National Research Foundation of Korea(NRF) grant funded by the Korea government(MSIT) (No. 2021R1C1C1005700)

\section{Outline of Proof}

In this section, we present the outline of the proof of our main result. First, in Section \ref{subsec: Reformulation of Boussinesq equation}, we reformulate the original  Boussinesq equation \eqref{Boussinesq} as a coupled system of nonlinear integral equations \eqref{Coupled Boussinesq} in a way that the Fourier analysis method can be carried out for our limit problem. Then, the main theorem (Theorem \ref{main theorem}) is rephrased accordingly (see Theorem \ref{main theorem'}). In Section \ref{subsec: Frequency localized decoupled Boussinesq equation}, we introduce an auxiliary equation, namely, the frequency localized decoupled Boussinesq equation \eqref{LFB}. It is a key tool to overcome a technical difficulty arising in application of the Fourier restriction norm to two different equations. Using this equation, we give a sketch of how we prove the theorem.

\subsection{Reformulation of the Boussinesq equation}\label{subsec: Reformulation of Boussinesq equation}

For a small parameter $\ep\in(0,1]$, we define
\begin{equation}\label{rescaled u}
u_\ep(t,x):=\frac{1}{\ep^2}u\bigg(\frac{t}{\ep^3}, \frac{x}{\ep}\bigg),
\end{equation}
where $u(t,x)$ is a solution to the Boussinesq equation \eqref{Boussinesq}. Then, $u_\ep(t,x)$ obeys the rescaled equation 
\begin{equation}\label{scaledBoussinesq}
\epsilon^4\pa_t^2 u_\ep= \pa_x^2(1-\epsilon^2\pa_x^2)u_\ep -\ep^2\pa_x^2(u_\ep^2),
\end{equation}
and the conservation law \eqref{energy} is also rescaled to 
\begin{equation}\label{rescaled energy}
E_\epsilon(u_\epsilon)=\frac{\ep^4}{2}\big\|\partial_x^{-1}\pa_t u_\epsilon\big\|_{L^2}^2+ \frac{\ep^2}{2}\|\partial_xu_\epsilon\|_{L^2}^2+\frac{1}{2}\|u_\epsilon\|_{L^2}^2-\frac{\ep^2}{3}\|u_\epsilon\|_{L^3}^3.
\end{equation}
By Duhamel's principle, the equation \eqref{scaledBoussinesq} is equivalent to the integral equation
\begin{align*}
u_\ep(t)
= \cos &\bigg( \frac{t}{\ep^2}|\partial_x|\langle\epsilon\partial_x\rangle\bigg)u_\epsilon(0) +\ep^2 \frac{\sin\left(\frac{t}{\ep^2}|\partial_x|\langle\epsilon\partial_x\rangle\right)}{|\partial_x|\langle\epsilon\partial_x\rangle}\partial_tu_\epsilon(0) \\
&+\int_0^t \sin\bigg(\frac{t-t_1}{\ep^2}|\partial_x|\langle\epsilon\partial_x\rangle\bigg)\frac{|\partial_x|}{\langle\epsilon\partial_x\rangle}(u_\ep(t_1))^2 dt_1,
\end{align*}
where $|\partial_x|$ and $\langle\epsilon\partial_x\rangle$ are the Fourier multipliers with symbol $|\xi|$ and $\langle\epsilon\xi\rangle$ respectively. By the decompositions 
\begin{align*}
\cos\bigg(\frac{t}{\ep^2}|\partial_x|\langle\epsilon\partial_x\rangle\bigg)&
=\frac12 e^{\frac{t}{\ep^2}\pa_x\langle\epsilon\partial_x\rangle}+\frac12 e^{-\frac{t}{\ep^2}\pa_x\langle\epsilon\partial_x\rangle},   \\
\sin\bigg(\frac{t}{\ep^2}|\partial_x|\langle\epsilon\partial_x\rangle\bigg)|\partial_x|&
=-\frac{1}{2}e^{\frac{t}{\ep^2}\pa_x\langle\epsilon\partial_x\rangle}\pa_x+\frac{1}{2}e^{-\frac{t}{\ep^2}\pa_x\langle\epsilon\partial_x\rangle}\pa_x, 
\end{align*}
the right hand side of the equation can be reorganized as 
$$\begin{aligned}
u_\ep(t)&=\frac12e^{\frac{t}{\ep^2}\pa_x\langle\epsilon\partial_x\rangle}u_\epsilon(0)+\frac12 e^{-\frac{t}{\ep^2}\pa_x\langle\epsilon\partial_x\rangle}u_\epsilon(0)\\
&\quad+\frac{1}{2} e^{\frac{t}{\ep^2}\pa_x\langle\epsilon\partial_x\rangle}\frac{\ep^2}{\langle\epsilon\partial_x\rangle}\partial_x^{-1}\partial_tu_\epsilon(0) 
-\frac{1}{2} e^{-\frac{t}{\ep^2}\pa_x\langle\epsilon\partial_x\rangle}\frac{\ep^2}{\langle\epsilon\partial_x\rangle}\partial_x^{-1}\partial_tu_\epsilon(0)\\
&\quad-\frac{1}{2}\int_0^t e^{\frac{t-t_1}{\ep^2}\pa_x\langle\epsilon\partial_x\rangle}\frac{\pa_x}{\langle\epsilon\partial_x\rangle}(u_\ep(t_1))^2dt_1
+\frac{1}{2}\int_0^t e^{-\frac{t-t_1}{\ep^2}\pa_x\langle\epsilon\partial_x\rangle}\frac{\pa_x}{\langle\epsilon\partial_x\rangle}(u_\ep(t_1))^2 dt_1.
\end{aligned}$$
Then, separating with respect to the linear propagators $e^{\pm\frac{t}{\ep^2}\pa_x\langle\epsilon\partial_x\rangle}$ and translating by $\pm\frac{t}{\epsilon^2}$, we set 
\begin{equation}\label{Coupled Boussinesq0}
u_{\ep}^{\pm}(t,x):=e^{\mp \frac{t}{\ep^2}\pa_x(\langle\epsilon\partial_x\rangle-1)} u_{\epsilon,0}^{\pm} \pm \frac{1}{2}\int_0^te^{\mp \frac{t-t_1}{\ep^2}\pa_x (\langle\epsilon\partial_x\rangle-1)}\frac{\pa_x}{\langle\epsilon\partial_x\rangle}e^{\pm\frac{t_1}{\epsilon^2}\partial_x}(u_\ep(t_1))^2 dt_1,
\end{equation}
where
\begin{equation}\label{initial data}
u_{\epsilon,0}^{\pm}=\frac12\bigg(u_\epsilon(0)\mp \frac{\ep^2}{\langle\epsilon\partial_x\rangle}\partial_x^{-1}\partial_tu_\epsilon(0)\bigg).
\end{equation}
Note that by definitions, 
\begin{equation}\label{decomposition relation}
u_\epsilon(t,x)=u_\ep^+\bigg(t,x-\frac{t}{\ep^2}\bigg)+u_\ep^-\bigg(t,x+\frac{t}{\ep^2}\bigg),
\end{equation}
since $e^{-a\partial_x}$ is a translation operator, i.e., $e^{-a\partial_x}f(x)=f(x-a)$. Therefore, introducing the operator
\begin{equation}\label{linear Boussinesq flow}
S_\ep^{\pm}(t)=e^{\mp \frac{t}{\ep^2}\pa_x(\langle\epsilon\partial_x\rangle-1)}=e^{\mp its_\epsilon(-i\partial_x)},
\end{equation}
where
\begin{equation}\label{Boussinesq phase function}
s_\ep (\xi)= \frac{\xi}{\ep^2}(\langle\epsilon\xi\rangle-1),
\end{equation}
and using that $e^{a\pa_x}(f^2)=f(x+a)^2=(e^{a\pa_x}f)^2$ for the nonlinear term, we may write \eqref{Coupled Boussinesq0} in a compact form as 
\begin{align}\label{Coupled Boussinesq}
\boxed{\quad u_{\ep}^{\pm}(t)=S_\ep^{\pm}(t) u_{\epsilon,0}^{\pm} \pm \frac{1}{2}\int_0^tS_\ep^{\pm}(t-t_1)\frac{\pa_x}{\langle\epsilon\partial_x\rangle}\Big\{ \big(u_\ep^\pm (t_1)+ e^{\pm \frac{2t_1}{\ep^2}\pa_x} u_\ep^\mp (t_1) \big)^2  \Big\} dt_1.\quad}
\end{align}
From now on, we call \eqref{Coupled Boussinesq} the \textit{coupled Boussinesq system}. 

\begin{remark}
Suppose that $(u_{\ep}^+, u_{\ep}^-)$ is a solution to the system \eqref{Coupled Boussinesq} in a proper function space (see Proposition \ref{Prop : local sol to Boussinesq}). Then, by construction together with \eqref{decomposition relation} and the well-posedness theory for the original Boussinesq equation \cite{Linares1993}, the counter-propagating two-wave $\epsilon^2u_\ep^+(\epsilon^3t, \epsilon^2(x-\frac{t}{\epsilon}))+\epsilon^2 u_\ep^-(\epsilon^3t, \epsilon^2(x+\frac{t}{\epsilon}))$ is a unique solution to the equation \eqref{Boussinesq} with given initial data.
\end{remark}

\begin{remark}\label{formal convergence remark}
\begin{enumerate}
\item The linear Boussinesq flow $S_\ep^{\pm}(t)$ formally converges to the Airy flow
\begin{equation}\label{Airy flow}
S^{\pm}(t) = e^{\pm \frac{t}{2}\partial_x^3},
\end{equation}
because $s_\ep (\xi)=\frac{\xi^3}{1+\langle\epsilon\xi\rangle}\to s(\xi)=\frac{\xi^3}{2}$ for each $\xi\in\mathbb{R}$. Note that in \eqref{Coupled Boussinesq0}, the spatial translation operator $e^{\pm\frac{t\partial_x}{\epsilon^2}}$ is inserted for this convergence by cancelling out the linear component in $\frac{t}{\ep^2}\pa_x\langle\epsilon\partial_x\rangle$. 
\item In \eqref{Coupled Boussinesq}, $u_{\ep}^{\pm}$ is coupled with $u_{\ep}^{\mp}$ via the term $e^{\pm \frac{2t_1}{\ep^2}\pa_x} u_\ep^\mp (t_1)$ in nonlinearity. However, this coupled term is expected to vanish as $\epsilon\to 0$. Indeed, if $u_{\ep}^{\pm}(t,x)\approx S_\ep^{\pm}(t) u_{\epsilon,0}^{\pm}$ at least on a short time interval (by the perturbative nature of the nonlinear equation), then  $e^{\pm \frac{2t}{\ep^2}\pa_x} u_\ep^\mp (t)$ can be approximated by
$$e^{\pm  \frac{t}{\ep^2}\pa_x(\langle\epsilon\partial_x\rangle+1)} u_{\epsilon, 0}^{\mp}(x)=\frac{1}{2\pi}\int_{\R}e^{\pm \frac{it}{\ep^2}\xi(\langle\epsilon\xi\rangle+1)}e^{ix\xi} \widehat{u_{\epsilon, 0}^{\mp}}(\xi)d\xi.$$
Then, the diverging group velocity $(\pm \frac{\xi}{\ep^2}(\langle\epsilon\xi\rangle+1))'=\pm\frac{1}{\ep^2}(\langle\epsilon\xi\rangle+1 + \frac{\xi^2}{\langle\epsilon\xi\rangle})\to\pm\infty$ as $\epsilon\to0$ would lead fast dispersion.
\item By (1) and (2), each $u_\epsilon^\pm(t,x)$ formally converges to the solution to the KdV equation 
\begin{equation}\label{integral KdV}
\boxed{\quad w_\epsilon^{\pm}(t) = S^\pm(t)u_{\epsilon,0}^{\pm} \pm \frac12 \int_0^t S^\pm(t-t_1)\partial_x (w_\epsilon^{\pm}(t_1))^2dt_1.\quad}
\end{equation}
\end{enumerate}
\end{remark}

By the reformulation of the problem (see \eqref{rescaled u}, \eqref{initial data} and \eqref{decomposition relation}), our main result in the introduction can be translated as follows. The rest of this paper is devoted to its proof. 

\begin{theorem}[Reformulated Korteweg-de Vries limit for the Boussinesq equation]\label{main theorem'}
Let $\ep\in(0,1]$, and suppose that $0<s\leq 5$ and 
$$\sup_{\epsilon\in(0,1]}\|u_{\epsilon,0}^{\pm}\|_{H^s}\leq R.$$
Then, there exist $T(R)>0$ and a unique solution $u_\ep^{\pm}\in C_t([-T,T]; H^s)$ to the coupled Boussinesq system \eqref{Coupled Boussinesq} with an initial data $u_{\epsilon,0}^{\pm}$ such that 
$$\|u_\ep^\pm(t)- w_\epsilon^\pm(t)\|_{C_t([-T,T]; L^2)} \lesssim_R\ep^{\min\{\frac{2s}{5}, \frac{1}{2}\}},$$
where $w_\ep^{\pm}\in C_t(\mathbb{R}; H^s)$ is a unique solution to the KdV equation \eqref{integral KdV} with the initial data $u_{\epsilon,0}^{\pm}$.
\end{theorem}

\begin{remark}
\begin{enumerate}
\item The constraint $s\leq 5$ is from the approximation lemma (Lemma \ref{Lem:difference linear sol in Xsb}), more essentially from the leading order term in the gap $s_\epsilon(\xi)-s(\xi)=-\frac{\epsilon^2\xi^5}{8}+\frac{\epsilon^4\xi^7}{16}\cdots$.
\item When $\frac{5}{4}\leq s\leq 5$, assuming more regularity on initial data does not improve the $O(\sqrt{\epsilon})$-rate of convergence. This is due to Lemma \ref{Lem:BL2}.
\end{enumerate}
\end{remark}

\subsection{Frequency localized decoupled Boussinesq equation}\label{subsec: Frequency localized decoupled Boussinesq equation}

Despite the positive aspects mentioned in Remark \ref{formal convergence remark}, we must pay attention to the high frequency contribution of solutions. Indeed, in this article, the Fourier restriction norm of Bourgain \cite{B-1993Sch,B-1993KdV} will be employed, because it is known to be a fundamental tool to capture smoothing properties from dispersion, in particular, for the low regularity theory, and our goal is to reduce the regularity requirement for the KdV limit. For $s,b\in \R$, the Fourier restriction norm associated with the Boussinesq flow $S_\epsilon^\pm(t)$ (resp., that associated with the Airy flow $S^\pm(t)$) is defined by
\begin{equation}\label{Fourier restriction norm}
\begin{aligned}
\|u\|_{X^{s,b}_{\ep,\pm}}:=\| \langle \xi \rangle^{s}\langle \tau\pm s_\ep(\xi)\rangle^b \tilde{u}(\tau,\xi)\|_{L_{\tau,\xi}^2}\quad\Big(\textup{resp., }\|u\|_{X^{s,b}_{\pm}}:=\| \langle \xi \rangle^{s}\langle \tau\pm s(\xi)\rangle^b \tilde{u}(\tau,\xi)\|_{L_{\tau,\xi}^2}\Big),
\end{aligned}
\end{equation}
where $\tilde{u}$ denotes the space-time Fourier transform of $u$, i.e., 
\begin{equation}\label{space-time Fourier transform}
\tilde{u}(\tau,\xi)=\iint u(t,x)e^{-i(t\tau+x\xi)} dxdt.
\end{equation}
However, a problem is that the Fourier restriction norm is very sensitive to the choice of the associated linear flow. In our situation, the linear Boussinesq flow $S_\ep^{\pm}(t)$ turns out to behave completely differently from the Airy flow $S^{\pm}(t)$ in high frequencies, because $s_\epsilon(\xi)\approx\frac{|\xi|\xi}{\epsilon}$ when $|\xi|\gg\frac{1}{\epsilon}$. For this reason, we are unable to directly measure the difference between \eqref{Coupled Boussinesq} and \eqref{integral KdV} either in the $X_{\epsilon,\pm}^{0,b}$-norm or in the $X_{\pm}^{0,b}$-norm (see also Remark \ref{high frequency worry}).

To overcome this obstacle, we introduce the following auxiliary equation 
\begin{equation}\label{LFB}
\boxed{\quad v_{\ep}^{\pm}(t)=S_\ep^{\pm}(t)P_{\leq N}u_{\epsilon,0}^{\pm} \pm \frac{1}{2}\int_0^tS_\ep^{\pm}(t-t_1)  P_{\leq N}\frac{\pa_x }{\langle\epsilon\partial_x\rangle}(P_{\leq N}v_\ep^\pm (t_1))^2   dt_1,\quad}
\end{equation}
where $P_{\le N}$ denotes the ``sharp'' frequency truncation operator defined by $\widehat{P_{\le N}f}(\xi)=\mathbbm{1}_{\{|\xi|\leq N\}}\hat{f}(\xi)$ and $N=\frac{1}{2}\epsilon^{-\frac{2}{5}}$. We call it the \textit{frequency localized decoupled Boussinesq equation}, since the coupled nonlinearity is dropped and $P_{\leq N}$ cuts out high frequencies. Note that if $v_\epsilon^\pm$ is a solution, then it is localized in $|\xi|\leq N$ on the Fourier side. Hence, we may write $(P_{\leq N}v_\ep^\pm (t_1))^2=(v_\ep^\pm (t_1))^2$ on the right hand side of \eqref{LFB}. 

By construction, it suffices to show that the following two differences converge to zero;
\begin{equation}\label{difference1}
\begin{aligned}
&(u_{\ep}^{\pm}-v_{\ep}^{\pm})(t)\\
&=S_\ep^{\pm}(t)(1-P_{\leq N})u_{\epsilon,0}^{\pm} &\textup{(high freq. linear flow)}\\
&\quad\pm \frac{1}{2}\int_0^tS_\ep^{\pm}(t-t_1)\frac{\pa_x}{\langle\epsilon\partial_x\rangle}\big\{(u_\ep^\pm+v_\ep^\pm)(u_\ep^\pm-v_\ep^\pm)\big\}(t_1)dt_1&\textup{(nonlinear difference)}\\
&\quad\pm \frac{1}{2}\int_0^tS_\ep^{\pm}(t-t_1)\frac{\pa_x}{\langle\epsilon\partial_x\rangle}\Big\{ 2u_\ep^\pm e^{\pm \frac{2t_1}{\ep^2}\pa_x}u_\ep^\mp+ e^{\pm \frac{2t_1}{\ep^2}\pa_x} (u_\ep^\mp)^2\Big\} (t_1)dt_1&\textup{(coupling nonlinearity)}
\end{aligned}
\end{equation}
and
\begin{equation}\label{difference2}
\begin{aligned}
&(v_{\ep}^{\pm}-w_{\ep}^{\pm})(t)\\
&=(1-P_{\leq N})w_{\ep}^{\pm}(t)&\textup{(high freq. KdV flow)}\\
&\quad+(S_\ep^{\pm}(t)-S^{\pm}(t))P_{\leq N}u_{\epsilon,0}^{\pm}&\textup{(linear low freq. error)}\\
&\quad\pm \frac{1}{2}\int_0^tS_\ep^{\pm}(t-t_1)\frac{\pa_x}{\langle\epsilon\partial_x\rangle}P_{\leq N}\big\{(v_\ep^\pm+w_\ep^\pm)(v_\ep^\pm-w_\ep^\pm)\big\}(t_1)dt_1&\textup{(nonlinear difference)}\\
&\quad\pm \frac{1}{2}\int_0^t(S_\ep^{\pm}(t-t_1)-S^{\pm}(t-t_1))\frac{\pa_x}{\langle\epsilon\partial_x\rangle}P_{\leq N}(w_\epsilon^{\pm}(t_1))^2dt_1&\textup{(nonlinear error 1)}\\
&\quad\mp \frac{1}{2}\int_0^tS^{\pm}(t-t_1)\partial_x\bigg( 1 - \frac{1 }{\langle\epsilon\partial_x\rangle} \bigg)P_{\leq N}(w_\epsilon^{\pm}(t_1))^2dt_1.&\textup{(nonlinear error 2)}
\end{aligned}
\end{equation}
For \eqref{difference1}, we take the $X_{\epsilon,\pm}^{0,b}$-norm. Then, the high frequency linear flow vanishes by the regularity gap; we measure it in $X_{\epsilon,\pm}^{0,b}$, assuming $\|u_{\epsilon,0}^{\pm}\|_{H^s}$ is uniformly bounded. On the other hand, the coupling nonlinearity converges to zero by fast dispersion. For \eqref{difference2}, the high frequency part of the KdV flow vanishes by the regularity gap. The (frequency localized) linear error and the nonlinear errors converge to zero because of symbol convergences in low frequencies. An important remark is that even though two flows $S_\ep^{\pm}(t)$ and $S^{\pm}(t)$ are present in \eqref{difference2}, we can take the $X_{\pm}^{0,b}$-norm, since the $X_{\epsilon,\pm}^{0,b}$- and the $X_{\pm}^{0,b}$-norms are equivalent in low frequencies (see Lemma \ref{Lem:frequency loc norm}). Finally, the nonlinear differences in \eqref{difference1} and \eqref{difference2} can be dealt with by proving proper uniform bounds (see Proposition \ref{Prop : local sol to Boussinesq}). 

\subsection{Outline of the paper}
The rest of this paper is organized as follows. In Section \ref{sec: linear estimate}, we review basic properties of the Fourier restriction norm and prove linear estimates for the convergence of the linear Boussinesq flow $S_\epsilon^\pm(t)$ in low frequencies. In Section \ref{sec: bilinear estimate}, we show uniform bilinear estimates to deal with the nonlinear terms in \eqref{Coupled Boussinesq} and \eqref{LFB}. Having these tools, in Section \ref{sec: uniform bounds for nonlinear solutions}, we establish well-posedness and uniform bounds for nonlinear solutions. Finally, in Section \ref{sec: proof of the main result}, we prove the main theorem.



\section{Linear estimates for the linear Boussinesq flow}\label{sec: linear estimate}

In this section, we prove some linear estimates in the Fourier restriction norms associated to the linear Boussinesq and the Airy flows (see \eqref{Fourier restriction norm} for the definitions). They will be used to show the convergence of the linear and the nonlinear errors in \eqref{difference2}.

\subsection{Basic properties of the Fourier restriction norm}

We recall the general theory on Fourier restriction norms (see \cite{Linares2015, Tao2006} for details). For a general symbol $p=p(\xi):\R\rightarrow\R$ and for $s,b\in \R$, we define the $X_{\tau=p(\xi)}^{s,b}$-space as the completion of $\mathcal{S}(\R^2)$ with respect to the norm
$$\|u\|_{X_{\tau=p(\xi)}^{s,b}}:=\left\| \langle \xi \rangle^{s}\langle \tau -  p(\xi)\rangle^b \tilde{u}(\tau,\xi)\right\|_{L_{\tau,\xi}^2(\R\times \R)},$$
where $\tilde{u}$ is the space-time Fourier transform (see \eqref{space-time Fourier transform}). 

\begin{lemma}[Properties of the Fourier restriction norm]\label{Lem:Basic property of Xsb}
Let $\eta\in C_c^\infty(\R)$. Then, the following hold for $s,b\in\R$ and $T\in(0,1]$.
\begin{enumerate}
\item (Nesting) If $s_1\leq s_2$ and $b_1\leq b_2$, then $X^{s_1,b_1}_{\tau=p(\xi)}\subseteq X^{s_2,b_2}_{\tau=p(\xi)}$. 
\item (Duality) $(X^{s,b}_{\tau=p(\xi)})^*=X^{-s,-b}_{\tau=-p(-\xi)}$.
\item (Embedding) For $b>\frac12$, $X_{\tau=p(\xi)}^{s,b}\subset C_t(\mathbb{R}; H_x^s)$.
\item (Linear flow in $X_{\tau=p(\xi)}^{s,b}$) If $b>\frac{1}{2}$, then $\|\eta(\frac{t}{T})e^{itp(-i\partial_x)} u_0\|_{X_{\tau=p(\xi)}^{s,b}} \lesssim T^{\frac12-b}\|u_0\|_{H^s}$.
\item (Stability with respect to time localization) If $-\frac12<b'\le b<\frac12$, then 
$$\|\eta(\tfrac{t}{T})u\|_{X_{\tau=p(\xi)}^{s,b'}} \lesssim T^{b-b'}\|u\|_{X_{\tau=p(\xi)}^{s,b}}.$$
If $\frac12<b\le1$, then
$$\|\eta(\tfrac{t}{T})u\|_{X_{\tau=p(\xi)}^{s,b}} \lesssim T^{\frac12-b}\|u\|_{X_{\tau=p(\xi)}^{s,b}}.$$
\item (Inhomogeneous term estimate) If $\frac12<b\le 1$, then
$$\bigg\|\eta(\tfrac{t}{T}) \int_0^te^{i(t-t_1)p(-i\partial_x)}F(t_1)dt_1\bigg\|_{X_{\tau=p(\xi)}^{s,b}} \lesssim T^{\frac12-b}\|F\|_{X_{\tau=p(\xi)}^{s,b-1}}.$$
\end{enumerate}
\end{lemma}

\subsection{Low frequency flows in the Fourier restriction norm}

An important remark is that choosing a right symbol is crucial when we use the Fourier restriction norm.
\begin{remark}\label{high frequency worry}
Consider the Fourier restriction norm of the linear Boussinesq flow with respect to the Airy flow, that is, 
\begin{equation}\label{Xsb for linear flow}
\begin{aligned}
\|\eta(t)S_\epsilon^\pm(t) u_0\|_{X_\pm^{s,b}}&=\big\|\langle\xi\rangle^s\langle \tau\pm s(\xi)\rangle^b\hat{\eta}(\tau\pm s_\epsilon(\xi))\hat{u}_0(\xi)\big\|_{L_{\tau,\xi}^2}\\
&=\big\|\langle\xi\rangle^s\langle \tau\pm  (s(\xi)-s_\epsilon(\xi))\rangle^b\hat{\eta}(\tau)\hat{u}_0(\xi)\big\|_{L_{\tau,\xi}^2},
\end{aligned}
\end{equation}
provided that $b>0$ and $u_0$ is merely in $H^s$. In spite of the convergence $s_\ep(\xi)\to s(\xi)$ in symbol (see Remark \ref{formal convergence remark} (1)), this quantity diverges as $\epsilon\to0$ due to the high frequency contribution, because $s(\xi)-s_\epsilon(\xi)\approx\frac{\xi^3}{2}-\frac{|\xi|\xi}{\epsilon}\approx \frac{\xi^3}{2}$ for $|\xi|\gg\frac{1}{\epsilon}$. 
\end{remark}

This may cause some technical issue when we directly measure the difference for two solutions $u_\epsilon^\pm$ and $w_\epsilon^\pm$ to \eqref{Coupled Boussinesq} and \eqref{integral KdV} via the Fourier restriction norm. Nevertheless, the following lemma asserts that the different Fourier restriction norms are comparable if we restrict to low frequencies.
\begin{lemma}\label{Lem:frequency loc norm}
Let $P_{\leq N}$ be the frequency cut-off given by \eqref{frequency truncation operator} with $N=\frac{1}{2}\epsilon^{-\frac{2}{5}}$. Then, for $0\leq\theta\leq 1$, we have the norm equivalence 
$$\| P_{\le N} u \|_{{X_{\pm}^{s,b}}} \sim \| P_{\le N} u \|_{X_{\tau=\mp(\theta s_\epsilon(\xi)+(1-\theta)s(\xi))}^{s,b}}.$$
\end{lemma}

\begin{proof}
We only consider $-$ case. Suppose that $|\xi|\leq N$. Then, by Taylor's theorem, we have
\begin{equation}\label{symbol difference}
|s_\ep(\xi)-s(\xi)| \le \tfrac18\ep^2|\xi|^5\leq 2^{-8},
\end{equation}
and thus, $\langle \tau\pm(\theta s_\ep(\xi)+(1-\theta)s(\xi)) \rangle= \langle \tau\pm s(\xi)\pm\theta(s_\ep(\xi)-s(\xi)) \rangle\sim \langle \tau\pm s(\xi)\rangle$. By definition, this proves the lemma.
\end{proof}

Moreover, in low frequencies, the linear Boussinesq flow can be approximated by that of the Airy flow in the Fourier restriction norm associated with the Airy flow.

\begin{lemma}\label{Lem:difference linear sol in Xsb}
Let $\eta\in C_c^\infty(\R)$ be a smooth bump function satisfying \eqref{bump function}, and let $P_N$ be the frequency cut-off given by \eqref{frequency truncation operator} with $N=\frac{1}{2}\epsilon^{-\frac{2}{5}}$. Then, for $0<T\leq1$, $0\leq s\leq 5$ and $\frac{1}{2}<b\leq 1$, we have 
\begin{equation}\label{homogenous term}
\|\eta(\tfrac{t}{T})( S_\ep^{\pm}(t) - S^{\pm}(t))P_{\le N}u_0\|_{X_{\pm}^{0,b}} \lesssim \ep^{\frac{2s}{5}}T^{\frac{3}{2}-b} \|u_0\|_{H^s}
\end{equation}
and
\begin{equation}\label{Inhomogenous term}
\bigg\|\eta(\tfrac{t}{T})\int_0^t (S_\ep^{\pm}(t-t_1)-S^{\pm}(t-t_1))\eta(\tfrac{t_1}{T})(P_{\le N} F)(t_1)dt_1 \bigg\|_{X_{\pm}^{0,b}} \lesssim \ep^{\frac{2s}{5}}T^{\frac{3}{2}-b}\|F\|_{X_{\pm}^{s,b-1}}.
\end{equation}
\end{lemma}

\begin{proof}
By the fundamental theorem of calculus, 
\begin{equation}\label{linear flow FTC}
e^{\mp its_\epsilon(\xi)}-e^{\mp its(\xi)}=\mp it(s_\epsilon(\xi)-s(\xi))\bigg\{\int_0^1 e^{\mp it(\theta s_\epsilon(\xi)+(1-\theta)s(\xi))}d\theta\bigg\}.
\end{equation}
Note that interpolating two inequalities in \eqref{symbol difference}, we have
\begin{equation}\label{symbol difference'}
|s_\ep(\xi)-s(\xi)| \lesssim \ep^\frac{2s}{5}|\xi|^s\quad\textup{for }|\xi|\leq N.
\end{equation}
Therefore, by the definition of the Fourier restriction norm, it follows from Lemma \ref{Lem:frequency loc norm}, \eqref{symbol difference'} and Lemma \ref{Lem:Basic property of Xsb}  that 
$$\begin{aligned}
&\|\eta(\tfrac{t}{T})( S_\ep^{\pm}(t) - S^{\pm}(t))P_{\le N}u_0\|_{X_{\pm}^{0,b}}\\
&\leq \int_0^1 \big\|t\eta(\tfrac{t}{T})e^{\mp it(\theta s_\epsilon(-i\partial_x)+(1-\theta)s(-i\partial_x))}(s_\epsilon(-i\partial_x)-s(-i\partial_x))P_{\le N}u_0\big\|_{X_{\pm}^{0,b}}d\theta\\
&\leq \ep^\frac{2s}{5}T\int_0^1 \big\|\tfrac{t}{T}\eta(\tfrac{t}{T})e^{\mp it(\theta s_\epsilon(-i\partial_x)+(1-\theta)s(-i\partial_x))}u_0\big\|_{X_{\tau=\mp(\theta s_\epsilon(\xi)+(1-\theta)s(\xi))}^{s,b}}d\theta \lesssim \ep^{\frac{2s}{5}}T^{\frac{3}{2}-b}\|u_0\|_{H^s}.
\end{aligned}$$
For \eqref{Inhomogenous term}, we use \eqref{linear flow FTC} to write 
$$\begin{aligned}
&\int_0^t (S^{\pm}(t-t_1)-S_\ep^{\pm}(t-t_1))\eta(\tfrac{t_1}{T})(P_{\le N} F)(t_1)dt_1\\ 
&=\mp i\int_0^t e^{\mp i(t-t_1)(\theta s_\epsilon(-i\partial_x)-(1-\theta)s(-i\partial_x))}(t-t_1)(s_\epsilon(-i\partial_x)-s(-i\partial_x))\eta(\tfrac{t_1}{T})(P_{\le N} F)(t_1)dt_1.
\end{aligned}$$
Then, as we proved \eqref{homogenous term}, using Lemma \ref{Lem:frequency loc norm}, \eqref{symbol difference'} and Lemma \ref{Lem:Basic property of Xsb}, one can show that the left hand side of \eqref{Inhomogenous term} is bounded by $\sim \ep^\frac{2s}{5}T^{\frac{3}{2}-b}\|P_{\leq N}F\|_{X_{\tau=\mp(\theta s_\epsilon(\xi)+(1-\theta)s(\xi))}^{s,b-1}}\lesssim\ep^\frac{2s}{5}T^{\frac{3}{2}-b}\|F\|_{X_{\pm}^{s,b-1}}$.
\end{proof}

\section{Bilinear estimates for the linear Boussinesq flow}\label{sec: bilinear estimate}

In this section, we prove various bilinear estimates associated with the linear Boussinesq flow, analogous to the well-known bilinear estimate for the Airy flow \cite{KPV1996};
\begin{equation*}
   \|\partial_x(w_1 w_2)\|_{X_{\pm}^{s,b-1}} \lesssim \|w_1\|_{X_{\pm}^{s,b}} \|w_2\|_{X_{\pm}^{s,b}}.
\end{equation*}
for $s>-\frac{3}{4}$ and for some $b\in(\frac12,1)$. 
In fact, for a nonnegative $s\ge 0$, the following improved bilinear estimate holds\footnote{See \cite[Lemma~3.14]{Erdogan2016}, for example.}
\begin{align}\label{bilinear KdV}
   \|\partial_x(w_1 w_2)\|_{X_{\pm}^{s,-\frac14}} \lesssim \|w_1\|_{X_{\pm}^{s,b}} \|w_2\|_{X_{\pm}^{s,b}}.
\end{align}
for all $b>\frac12$. The bilinear estimates for the Boussinesq flow will be used to obtain uniform bounds for the nonlinear solutions $u_{\ep}^{\pm}$ and $v_{\ep}^{\pm}$ and to prove that the coupled nonlinearity vanishes. 

To begin with, we collect some elementary facts which will be used frequently. 

\begin{remark}\label{derivatives of s}
The derivatives of the symbol $s_\ep(\xi)=\frac{\xi}{\ep^2}(\langle\epsilon\xi\rangle-1)=\frac{\xi^3}{1+\langle\epsilon\xi\rangle}$ are given by 
$$s_\ep'(\xi)=\frac{1}{\epsilon^2}(\langle\ep\xi\rangle-1)+\frac{\xi^2}{\langle\epsilon\xi\rangle^2},\quad 
s_\ep''(\xi)=\frac{\xi(1+2\langle\epsilon\xi\rangle^2) }{\langle\epsilon\xi\rangle^3},\quad s_\ep'''(\xi)=\frac{3}{\langle\epsilon\xi\rangle^5}.$$
\end{remark}
\begin{lemma}\label{Lem:tau integral}
For $b'\ge b>\frac12$, we have
$$\int_{\R}\frac{dx}{\langle x-\alpha\rangle^{2b'}{\langle x-\beta\rangle^{2b}}}\lesssim \frac{1}{\langle \alpha-\beta\rangle^{2b}},\quad \int_{\R} \frac{dx}{\langle x\rangle^{2b}\sqrt{|x-y|}}\lesssim \frac{1}{\langle y\rangle^{\frac12}}.$$
\end{lemma}

First, we prove the following bilinear estimate for the main nonlinear terms $\frac{\pa_x}{\langle\epsilon\partial_x\rangle}\big(u_\ep^\pm)^2$ and $\frac{\pa_x}{\langle\epsilon\partial_x\rangle}\big(v_\ep^\pm)^2$ in the equations \eqref{Coupled Boussinesq} and \eqref{LFB}. 

\begin{lemma}[Bilinear estimate I]\label{Lem:BL1}
For $s\ge0$ and $b>\frac12$, we have
$$\bigg\| \frac{\pa_x}{\langle\epsilon\partial_x\rangle}(uv)\bigg\|_{X_{\ep,\pm}^{s,-\frac{1}{4}}}
\ls \|u\|_{X_{\ep,\pm}^{s,b}} \|v\|_{X_{\ep,\pm}^{s,b}}.$$
\end{lemma}

\begin{proof}
\textbf{(Step 1. Reduction)}
We only prove the proposition for $\|\frac{\pa_x}{\langle\epsilon\partial_x\rangle}(uv)\|_{X_{\ep,+}^{s,-\frac{1}{4}}}$, because the other can be proved by the same way. Indeed, by duality, it suffices to show that 
$$\bigg| \iint \bigg(\frac{\pa_x}{\langle\ep\partial_x\rangle}(uv)\bigg)\bar{w}dxdt\bigg| \lesssim \| u\|_{X_{\ep,+}^{s,b}} \|v\|_{X_{\ep,+}^{s,b}}\|w\|_{X_{\ep,+}^{-s,\frac{1}{4}}}.$$
By Parseval's identity, the left hand side integral can be written as 
$$\frac{1}{(2\pi)^4}\iint\iint\frac{i\xi}{\langle \ep\xi\rangle} \tilde{u}(\tau_1, \xi_1)\tilde{v}(\tau-\tau_1, \xi-\xi_1) \overline{\tilde{w}(\tau,\xi)} d\xi_1 d\tau_1 d\xi d\tau.$$
Hence, by the definition of the Fourier restriction norm, the above trilinear estimate is equivalent to
$$\begin{aligned}
&\bigg| \iint\iint \frac{\xi}{\langle\ep\xi\rangle}\frac{\langle\xi\rangle^s}{\langle\xi_1\rangle^s\langle\xi-\xi_1\rangle^s}\frac{U(\tau_1, \xi_1)V(\tau-\tau_1, \xi-\xi_1)\overline{W(\tau,\xi)}}{\langle \tau_1+s_\ep(\xi_1)\rangle^b\langle \tau-\tau_1+s_\ep(\xi-\xi_1)\rangle^b\langle \tau+s_\ep(\xi)\rangle^{\frac{1}{4}}}d\xi_1d\tau_1d\xi d\tau\bigg| \\ 
&\lesssim \|U\|_{L_{\tau,\xi}^2}\|V\|_{L_{\tau,\xi}^2}\|W\|_{L_{\tau,\xi}^2}.
\end{aligned}$$
For this inequality, one can see that by the Cauchy-Schwarz inequality, the trivial inequality $\langle\xi\rangle\lesssim \langle\xi_1\rangle\langle\xi-\xi_1\rangle$ and Lemma \ref{Lem:tau integral} for the $\tau_1$-integration, the proof of the proposition can be reduced to show that
$$I_\ep(\tau,\xi):=\frac{\xi^2}{\langle\ep\xi\rangle^2\la \tau+s_\ep(\xi)\ra^{\frac{1}{2}}}\int_{\R}\frac{d\xi_1}{\la \tau +s_\epsilon(\xi_1)+s_\epsilon(\xi-\xi_1)\rangle^{2b}}$$
is bounded uniformly in $\tau, \xi$. Here, we may assume that $\xi>0$, since $I_\ep(\tau,-\xi) = I_\ep(-\tau,\xi)$ by the change of variable $\xi\mapsto -\xi$ with $s_\ep(\xi)=-s_\ep(-\xi)$. We also note that the integrand in $I_\ep(\tau,\xi)$ is symmetric with respect to $\xi_1=\frac{\xi}{2}$. Thus, we may write 
\begin{equation}\label{reduced integral estimate}
I_\ep(\tau,\xi)=\frac{2\xi^2}{\langle\ep\xi\rangle^2\la \tau+s_\ep(\xi)\ra^{\frac{1}{2}}}\int_{\frac{\xi}{2}}^\infty\frac{d\xi_1}{\la \tau +s_\epsilon(\xi_1)+s_\epsilon(\xi-\xi_1)\rangle^{2b}}.
\end{equation}

\medskip

\noindent
\textbf{(Step 2. Substitution)}
For the integral in \eqref{reduced integral estimate}, we assume that $\xi_1\geq\frac{\xi}{2}$ and substitute 
$$z=s_\epsilon(\xi_1)+s_\epsilon(\xi-\xi_1)=-\frac{\xi}{\epsilon^2}+\frac{\xi_1}{\ep^2}{\langle \ep\xi_1\rangle}+   \frac{\xi-\xi_1}{\ep^2}{\langle \ep(\xi-\xi_1)\rangle}$$
(see Remark \ref{derivatives of s}). 
Indeed, direct computations, involving rationalization of $\langle \ep\xi_1\rangle-\langle \ep(\xi-\xi_1)\rangle$, yield 
\begin{equation}\label{derivative of z}
\begin{aligned}
z'(\xi_1)&=\frac{\langle \ep\xi_1\rangle}{\ep^2}+\frac{\xi_1^2}{\langle \ep\xi_1\rangle}- \frac{\langle \ep(\xi-\xi_1)\rangle}{\ep^2}-\frac{(\xi-\xi_1)^2}{\langle \ep(\xi-\xi_1)\rangle}\\
&=\frac{\langle \ep\xi_1\rangle-\langle \ep(\xi-\xi_1)\rangle}{\ep^2}+\frac{\xi_1^2-(\xi-\xi_1)^2}{\langle \ep\xi_1\rangle}+(\xi-\xi_1)^2\bigg(\frac{1}{\langle \ep\xi_1\rangle}-\frac{1}{\langle \ep(\xi-\xi_1)\rangle}\bigg)\\
&=\xi(2\xi_1-\xi)\frac{2\langle\ep\xi_1\rangle \langle \ep(\xi-\xi_1)\rangle+1}{\langle\ep\xi_1\rangle \langle \ep(\xi-\xi_1)\rangle \big( \langle\ep\xi_1\rangle + \langle \ep(\xi-\xi_1)\rangle\big)}\geq\frac{\xi(2\xi_1-\xi)}{\langle\ep\xi_1\rangle},
\end{aligned}
\end{equation}
because $\langle\ep(\xi-\xi_1)\rangle \leq\langle \ep\xi_1\rangle$ when $\xi_1\geq\frac{\xi}{2}$. It implies that $z(\xi_1)$ is strictly increasing on $[\frac{\xi}{2},\infty)$ so that we can make such a change of variable. Hence, it follows that 
$$\begin{aligned}
I_\ep(\tau,\xi)&=\frac{2\xi^2}{\langle\ep\xi\rangle^2\la \tau+s_\ep(\xi)\ra^{\frac{1}{2}}}\int_{z(\frac{\xi}{2})}^\infty\frac{dz}{\la \tau+z\ra^{2b}z'(\xi_1)}\\
&\leq \frac{2\xi}{\langle\ep\xi\rangle^2\la \tau+s_\ep(\xi)\ra^{\frac{1}{2}}}\int_{z(\frac{\xi}{2})}^{z(\xi+\frac{1}{\epsilon})}+\int_{z(\xi+\frac{1}{\epsilon})}^\infty\frac{\langle\ep\xi_1\rangle}{\la \tau+z\ra^{2b}(2\xi_1-\xi)}dz\\
&=:I_{\ep}^-(\tau,\xi)+I_{\ep}^+(\tau,\xi).
\end{aligned}$$
\medskip
\noindent
\textbf{(Step 3. Integral estimates)}
It remains to obtain uniform bounds for $I_\epsilon^\pm(\tau,\xi)$ assuming $\xi>0$. For $I_\epsilon^+(\tau,\xi)$, we note that if $\xi_1\geq\xi+\frac{1}{\epsilon}$,
\begin{equation}\label{high frequency integral proof}
2\xi_1-\xi=\xi_1+(\xi_1-\xi)\gtrsim\frac{\langle\epsilon\xi_1\rangle}{\epsilon}.
\end{equation}
Hence, it follows that 
$$I_\ep^+(\tau,\xi)\lesssim\frac{\epsilon\xi}{\langle\ep\xi\rangle^2}\int_{\mathbb{R}}\frac{dz}{\la \tau+z\ra^{2b}}\lesssim 1.$$
To estimate $I_{\ep}^-(\tau,\xi)$, we assume that $0\leq\frac{\xi}{2}\leq\xi_1\leq\xi+\frac{1}{\epsilon}$. Then, $0\leq\epsilon\xi_1\leq 1+\epsilon\xi$ and $-\frac{\epsilon\xi}{2}\leq\epsilon(\xi_1-\xi)\leq 1$. Thus, we have
\begin{equation}\label{low frequency integral proof 1}
\langle\ep\xi_1\rangle + \langle \ep(\xi-\xi_1)\rangle\lesssim \langle\ep\xi\rangle.
\end{equation}
On the other hand, it follows from Taylor's theorem to $z(\xi_1)$ near $\xi_1=\frac{\xi}{2}$ with $z'(\frac{\xi}{2})=0$ (see \eqref{derivative of z}) that there exists $\frac{\xi}{2}<\xi_*<\xi_1$ such that $z(\xi_1)=z(\tfrac{\xi}{2}) +\tfrac12z''(\xi_*)(\xi_1-\tfrac{\xi}{2})^2$ and $z''(\xi_*)>0$, in other words,
\begin{equation}\label{low frequency integral proof 2}
\frac{1}{2\xi_1-\xi}=\frac{(z''(\xi_*))^{\frac{1}{2}}}{2\sqrt{2}(z(\xi_1)-z(\frac{\xi}{2}))^{\frac{1}{2}}}.
\end{equation}
We claim that if $\xi_1>\frac{\xi}{2}$, then 
\begin{equation}\label{low frequency integral proof 3}
z''(\xi_1)\lesssim \frac{\xi}{\langle \ep\xi\rangle}.
\end{equation}
Indeed, by the derivative formulae for $s_\epsilon(\xi)$ (see Remark~\ref{derivatives of s}), $s_\ep''$ is odd and increasing, and $0\leq s_\ep'''\leq3$. Thus, by the mean-value theorem, it follows that 
$$z''(\xi_1)= s_\ep''(\xi_1) + s_\ep''(\xi-\xi_1)=s_\ep''(\xi_1) - s_\ep''(\xi_1-\xi)\leq3\xi$$
and
$$z''(\xi_1)=s_\ep''(\xi_1) + s_\ep''(\xi-\xi_1)\leq 2s_\ep''(\xi_1)=\frac{\xi_1(3+2\ep^2\xi_1^2) }{\langle\epsilon\xi_1\rangle^3} \leq \frac{3}{\ep}.$$
Combining two upper bounds, we prove the claim. Now, we apply \eqref{low frequency integral proof 1}, \eqref{low frequency integral proof 2} and \eqref{low frequency integral proof 3} and Lemma \ref{Lem:tau integral} to $I_{\ep}^-(\tau,\xi)$. Then, it follows that 
$$\begin{aligned}
I_{\ep}^-(\tau,\xi)&\lesssim \frac{\xi^{\frac{3}{2}}}{\langle\ep\xi\rangle^{\frac{3}{2}}\la \tau+s_\ep(\xi)\ra^{\frac{1}{2}}}\int_{\mathbb{R}}\frac{dz}{\la \tau+z\ra^{2b}(z-z(\frac{\xi}{2}))^{\frac{1}{2}}}\\
&\lesssim\frac{\xi^{\frac{3}{2}}}{\langle\ep\xi\rangle^{\frac{3}{2}}\la \tau+s_\ep(\xi)\ra^{\frac{1}{2}}\la \tau+z(\frac{\xi}{2})\ra^{\frac{1}{2}}}\lesssim\frac{\xi^{\frac{3}{2}}}{\langle\ep\xi\rangle^{\frac{3}{2}}\la s_\ep(\xi)-z(\frac{\xi}{2})\ra^{\frac{1}{2}}}\sim1,
\end{aligned}$$
where in the last step, we used that $s_\ep(\xi)-z(\frac{\xi}{2})=\frac{\xi}{\epsilon^2}(\langle\epsilon\xi\rangle-\langle\frac{\epsilon\xi}{2}\rangle)=\frac{3\xi^3}{4(\langle\epsilon\xi\rangle+\langle\frac{\epsilon\xi}{2}\rangle)}\sim\frac{\xi^3}{\langle\epsilon\xi\rangle}$.
\end{proof}

The next two bilinear estimates are for the coupling nonlinearities $u_\ep^\pm\cdot e^{\pm \frac{2t}{\ep^2}\pa_x} u_\ep^\mp$ and $(e^{\pm \frac{2t}{\ep^2}\pa_x} u_\ep^\mp)^2=e^{\pm \frac{2t}{\ep^2}\pa_x}(u_\ep^\mp)^2$ in the coupled Boussinesq system \eqref{Coupled Boussinesq}. Here, we also consider the case when higher regularity norms are present on the right hand side $(s'>s)$ so that we can make use of the estimates to show that the coupling nonlinearities by the regularity gap. 

\begin{lemma}[Bilinear estimate II]\label{Lem:BL2}
If $s'\geq s\geq 0$ and $b>\frac12$, then 
$$\bigg\| \frac{\pa_x}{\langle\epsilon\partial_x\rangle}e^{\pm \frac{2t}{\ep^2}\pa_x}(uv)\bigg\|_{X_{\ep,\pm}^{s,-\frac{1}{4}}} \ls  \ep^{\min\{s'-s,\frac{1}{2}\}}\|u\|_{X_{\ep,\mp}^{s',b}} \|v\|_{X_{\ep,\mp}^{s',b}}.$$
\end{lemma}
   
\begin{proof}
We only prove for $\| \frac{\pa_x}{\langle\epsilon\partial_x\rangle}e^{\frac{2t}{\ep^2}\pa_x}(uv)\|_{X_{\ep,+}^{s,-\frac{1}{4}}}$. Indeed, repeating the reduction step in the proof of Lemma \ref{Lem:BL1} line by line but using that $\frac{\langle\xi\rangle^{2s}}{\langle\xi_1\rangle^{2s'}\langle\xi-\xi_1\rangle^{2s'}}\lesssim\frac{1}{\langle\xi_1\rangle^{2(s-s')}}$ for $\xi_1\geq\frac{\xi}{2}$, one can reduce the proof of the proposition to show that 
$$\sup_{\xi>0,\tau\in\R}\frac{2\xi^2}{\la \ep\xi\ra^2\la \tau + s_\ep(\xi)\ra^{\frac{1}{2}}}\int_{\frac{\xi}{2}}^{\infty}\frac{d\xi_1}{\la \xi_1\ra^{2(s'-s)}\la \tau- \frac{2\xi}{\ep^2}-s_\ep(\xi_1)-s_\ep(\xi-\xi_1)\ra^{2b}}\ls \ep^{\min\{2(s'-s),1\}}.$$
Note that compared to \eqref{reduced integral estimate}, in $\langle\cdots\rangle^{2b}$, the additional term $-\frac{2\xi}{\epsilon}$ comes from the translation operator $e^{\pm \frac{2t}{\ep^2}\pa_x}$, while $s_\epsilon(\xi_1)$ and $s_\epsilon(\xi-\xi_1)$ have different signs, because $X_{\epsilon,-}^{s,b}$-norms are taken on the right hand side. Suppose that $\xi>0$, and substitute 
$$z(\xi_1)=s_\ep(\xi_1)+s_\ep(\xi-\xi_1)$$
as before. Then, it follows from \eqref{derivative of z} that  
$$\begin{aligned}
&\frac{2\xi^2}{\la \ep\xi\ra^2\la \tau + s_\ep(\xi)\ra^{\frac{1}{2}}}\int_{\frac{\xi}{2}}^{\infty}\frac{d\xi_1}{\la \xi_1\ra^{2(s'-s)}\la \tau- \frac{2\xi}{\ep^2}-s_\ep(\xi_1)-s_\ep(\xi-\xi_1)\ra^{2b}}\\
&=\frac{2\xi^2}{\la \ep\xi\ra^2\la \tau + s_\ep(\xi)\ra^{\frac{1}{2}}}\int_{z(\frac{\xi}{2})}^\infty\frac{dz}{\la \xi_1\ra^{2(s'-s)}\la \tau- \frac{2\xi}{\ep^2}-z\ra^{2b}z'(\xi_1)}\\
&\leq \frac{2\xi}{\la \ep\xi\ra^2\la \tau + s_\ep(\xi)\ra^{\frac{1}{2}}}\int_{z(\frac{\xi}{2})}^{z(\xi+\frac{1}{\epsilon})}+\int_{z(\xi+\frac{1}{\epsilon})}^\infty\frac{\langle\ep\xi_1\rangle}{\la \xi_1\ra^{2(s'-s)}\la \tau- \frac{2\xi}{\ep^2}-z\ra^{2b}(2\xi_1-\xi)}dz\\
&=:I_{\ep}^-(\tau,\xi)+I_{\ep}^+(\tau,\xi).
\end{aligned}$$
For $I_{\ep}^+(\tau,\xi)$, it follows from \eqref{high frequency integral proof} that
$$\begin{aligned}
I_{\ep}^+(\tau,\xi)&\lesssim \frac{\epsilon\xi}{\langle\ep\xi\rangle^2\la \tau + s_\ep(\xi)\ra^{\frac{1}{2}}}\int_{z(\xi+\frac{1}{\epsilon})}^\infty\frac{dz}{\la \xi_1\ra^{2(s'-s)}\la \tau- \frac{2\xi}{\ep^2}-z\ra^{2b}}\\
&\lesssim \frac{\epsilon\xi}{\langle\ep\xi\rangle^2\la \xi+\frac{1}{\epsilon}\ra^{2(s'-s)}}\leq\epsilon^{2(s'-s)}\frac{\epsilon\xi}{\langle\ep\xi\rangle^{2(1+s'-s)}}\lesssim \epsilon^{2(s'-s)}.
\end{aligned}$$
For $I_{\ep}^-(\tau,\xi)$, we employ \eqref{low frequency integral proof 1}, \eqref{low frequency integral proof 2} and \eqref{low frequency integral proof 3} and Lemma \ref{Lem:tau integral} to obtain 
$$\begin{aligned}
I_{\ep}^-(\tau,\xi)&\lesssim\frac{\xi^{\frac{3}{2}}}{\la \xi\ra^{2(s'-s)}\la \ep\xi\ra^{\frac{3}{2}}\la \tau + s_\ep(\xi)\ra^{\frac{1}{2}}}\int_{z(\frac{\xi}{2})}^{z(\xi+\frac{1}{\epsilon})}\frac{dz}{\la \tau- \frac{2\xi}{\ep^2}-z\ra^{2b}(z-z(\frac{\xi}{2}))^{\frac{1}{2}}}\\
&\lesssim\frac{\xi^{\frac{3}{2}}}{\la \xi\ra^{2(s'-s)}\la \ep\xi\ra^{\frac{3}{2}}\la \tau + s_\ep(\xi)\ra^{\frac{1}{2}}\la \tau- \frac{2\xi}{\ep^2}-z(\frac{\xi}{2})\rangle^{\frac{1}{2}}}\lesssim\frac{\xi^{\frac{3}{2}}}{\la \xi\ra^{2(s'-s)}\la \ep\xi\ra^{\frac{3}{2}}\la \frac{2\xi}{\ep^2}+z(\frac{\xi}{2}) + s_\ep(\xi)\ra^{\frac{1}{2}}}\\
&\lesssim \frac{\xi^{\frac{3}{2}}}{\la \xi\ra^{2(s'-s)}\la \ep\xi\ra^{\frac{3}{2}}\la \frac{\xi}{\ep^2}\langle\ep\xi\rangle\ra^{\frac{1}{2}}}=\frac{\epsilon\xi}{\langle\xi\rangle^{2(s'-s)}\langle\epsilon\xi\rangle^2},
\end{aligned}$$
where in the last inequality, we used that $\frac{2\xi}{\ep^2}+z(\frac{\xi}{2}) + s_\ep(\xi)=\frac{\xi}{\ep^2}\langle\tfrac{\ep\xi}{2}\rangle+\frac{\xi}{\ep^2}\langle\ep\xi\rangle$. Hence, if $\xi\leq\frac{1}{\epsilon}$, then $I_{\ep}^-(\tau,\xi)\sim\frac{\epsilon\xi}{\langle\xi\rangle^{2(s'-s)}}\leq\epsilon^{\min\{2(s'-s),1\}}$. On the other hand, if $\xi\geq\frac{1}{\epsilon}$, then $I_{\ep}^-(\tau,\xi)\sim \frac{1}{\epsilon\xi^{2(s'-s)+1}}\leq\epsilon^{2(s'-s)}$.
\end{proof}

\begin{lemma}[Bilinear estimate III]\label{Lem:BL3}
If $s'\geq s\geq 0$ and $b>\frac12$, then 
$$\bigg\| \frac{\pa_x}{\langle\epsilon\partial_x\rangle}\Big(u\cdot e^{\pm \frac{2t}{\ep^2}\pa_x}v\Big)\bigg\|_{X_{\ep,\pm}^{s,-\frac{1}{4}}} \ls  \ep^{\min\{s'-s,1\}}\|u\|_{X_{\ep,\pm}^{s',b}} \|v\|_{X_{\ep,\mp}^{s',b}}.$$
\end{lemma}

\begin{proof}
We only prove that $\| \frac{\pa_x}{\langle\epsilon\partial_x\rangle}(u\cdot e^{\frac{2t}{\ep^2}\pa_x}v)\|_{X_{\ep,+}^{s,-\frac{1}{4}}} \ls  \ep^{s'-s}\|u\|_{X_{\ep,+}^{s',b}} \|v\|_{X_{\ep,-}^{s',b}},$
which can be reduced as in the proofs of the previous two lemmas to the following integral estimate
$$I_{\tau,\xi}:=\frac{2\xi^2}{\la \xi\ra^{2(s'-s)}\la \ep\xi\ra^2\la \tau + s_\ep(\xi)\ra^{\frac{1}{2}}}\int_{\frac{\xi}{2}}^{\infty}\frac{1}{\la \tau- \frac{2(\xi-\xi_1)}{\ep^2}+s_\ep(\xi_1)-s_\ep(\xi-\xi_1)\ra^{2b}}d\xi_1\ls \ep^{2\min\{s'-s,1\}}$$
for $\xi>0$ and $\tau\in\mathbb{R}$. Indeed, changing the variable for the integral by
$$z(\xi_1) = \frac{2\xi_1-\xi}{\ep^2}+s_\ep(\xi_1)-s_\ep(\xi-\xi_1)=\frac{\xi_1}{\ep^2}\langle\epsilon\xi_1\rangle-\frac{\xi-\xi_1}{\ep^2}\langle\epsilon(\xi-\xi_1)\rangle$$
with
$$z'(\xi_1)=\frac{\langle\epsilon\xi_1\rangle}{\ep^2}+\frac{\langle\epsilon(\xi-\xi_1)\rangle}{\ep^2}+\frac{\xi_1^2}{\langle\epsilon\xi_1\rangle}+\frac{(\xi-\xi_1)^2}{\langle\epsilon(\xi-\xi_1)\rangle}\geq\frac{\langle\frac{\epsilon\xi}{2}\rangle}{\epsilon^2},$$
it follows that
$$I_{\tau,\xi}\lesssim \frac{\xi^2}{\la \xi\ra^{2(s'-s)}\la \ep\xi\ra^2 }\frac{\epsilon^2}{\langle\frac{\epsilon\xi}{2}\rangle}\int_\mathbb{R}\frac{dz}{\la \tau-\frac{\xi}{\epsilon^2}+z\rangle^{2b}}\lesssim \frac{\epsilon^2\xi^2}{\la \xi\ra^{2(s'-s)}\la \ep\xi\ra^3}\ls \ep^{2\min\{s'-s,1\}}.$$
\end{proof}

\section{Uniform bounds for nonlinear solutions}\label{sec: uniform bounds for nonlinear solutions}

We recall the well-known well-posedness result for the KdV equation.

\begin{proposition}[Local well-posedness of the KdV equation \cite{KPV-1993DUKE, KPV1996}]\label{Pro : local sol to KdV}
Let $s>-\frac{3}{4}$. There exists $b\in(\frac12,1)$ such that for $u_0^{\pm}\in H^s$, there exist $T=T(\|u_0^{\pm}\|_{H^s})>0$ and a unique solution $w^{\pm}\in C_t([-T,T]:H^s)$ to the KdV equation \eqref{integral KdV} with an initial data $u_0^{\pm}$ such that $\|w^\pm\|_{X_{\pm}^{s,b}}\lesssim \|u_0^{\pm}\|_{H^s}$.
\end{proposition}

The goal of this section is to establish analogous well-posedness for the coupled/decoupled Boussinesq systems \eqref{Coupled Boussinesq} and \eqref{LFB}. Indeed, it is already shown in Linares \cite{Linares1993} that the original Boussinesq equation \eqref{Boussinesq} is locally well-posed in $H^s$ for $s\ge0$. However, this result or the argument of its proof does not directly provide what we need for the rescaled models, because its existence time $[-T_\epsilon,T_\epsilon]$ from the previous result may depend on the parameter $\epsilon>0$ and there is a possibility that this interval shrinks to zero in the limit. To see this, we note that the nonlinear term $\frac{\pa_x}{\langle\epsilon\partial_x\rangle}(u_\ep^\pm+ e^{\pm \frac{2t}{\ep^2}\pa_x} u_\ep^\mp)^2$ in \eqref{Coupled Boussinesq} includes a differential operator $\frac{\pa_x}{\langle\epsilon\partial_x\rangle}$ of order 0. However, this differential formally converges to $\partial_x$ in the limit $\epsilon\to 0$. Therefore, in order to establish well-posedness with $\epsilon$-uniformity, we need one derivative gain in essence. For this reason, uniform linear/bilinear estimates analogous to those for the KdV equation are prepared in the previous sections.

In this section, we prove the following well-posedness with $\epsilon$-uniformity.

\begin{proposition}[Local well-posedness and uniform bounds for the Boussinesq systems]\label{Prop : local sol to Boussinesq}
Let $s\ge0$ and $\frac12<b<\frac{5}{8}$, and suppose that 
$$\sup_{\epsilon\in(0,1]}\|u_{\epsilon,0}^{\pm}\|_{H^s}\leq R.$$
There exist $T(R)>0$, independent of $\ep\in(0,1]$, and a unique solution $u_\ep^{\pm}(t)$ (resp., $v_\ep^{\pm}(t)$) in $\in C_t([-T,T]; H^s)$ to the coupled Boussinesq equation \eqref{Coupled Boussinesq} (resp., the frequency localized decoupled Boussinesq equation \eqref{LFB}) with an initial data $u_{\epsilon, 0}^{\pm}$ (resp., $P_{\leq N}u_{\epsilon, 0}^\pm$) such that 
$$\sup_{\epsilon\in(0,1]}\|u_{\ep}^\pm\|_{X_{\ep,\pm}^{s,b}} \lesssim R\quad\textup{and}\quad\sup_{\epsilon\in(0,1]}\|v_{\ep}^\pm\|_{X_{\ep,\pm}^{s,b}} \lesssim R.$$
\end{proposition}

\begin{proof}
We only give a proof for the coupled Boussinesq equation \eqref{Coupled Boussinesq}. Indeed, the proof for the other equation \eqref{Coupled Boussinesq} follows by the same way, because it has the same structure with less nonlinear terms and the frequency truncation $P_{\leq N}$ only makes terms smaller. 

For notational convenience, we denote $\mathbf{X}_{\ep}^{s,b}:=X_{\ep,+}^{s,b}\times X_{\ep,-}^{s,b}$ and $\mathbf{H}^{s}=H^{s}\times H^{s}$. Fix initial $\mathbf{u}_{\epsilon,0}=(u_{\epsilon,0}^{+},u_{\epsilon,0}^{-})$ such that $\|\mathbf{u}_{\epsilon,0}\|_{\mathbf{H}^s}\leq 2R$, and let $T>0$ be a small number to be chosen later. For $\mathbf{u}=(u^+,u^-)\in \mathbf{X}_{\ep}^{s,b}$, we define 
$$\mathbf{\Phi}_\epsilon(\mathbf{u})=(\Phi_\epsilon^+(\mathbf{u}), \Phi_\epsilon^-(\mathbf{u}))$$
by 
$$\Phi_\epsilon^\pm(\mathbf{u}) = \eta_T(t)S_\ep^{\pm}(t) u_{\epsilon,0}^{\pm} \pm \frac{\eta_T(t)}{2}\int_0^tS_\ep^{\pm}(t-t_1)\eta_T(t_1)\frac{\pa_x}{\langle\epsilon\partial_x\rangle}\Big\{ \big(u_\ep^\pm (t_1)+ e^{\pm \frac{2t_1}{\ep^2}\pa_x} u_\ep^\mp (t_1) \big)^2  \Big\} dt_1,$$
where $\eta$ is a smooth bump function satisfying  \eqref{bump function} and $\eta_T=\eta(\frac{\cdot}{T})$. Then, it follows from Lemma~\ref{Lem:Basic property of Xsb} that 
\begin{equation}\label{contraction}
\begin{aligned}
\|\Phi_\epsilon^\pm(\mathbf{u})\|_{X_{\ep,\pm}^{s,b}}
   &\lesssim \|u_{\epsilon,0}^{\pm}\|_{H^s} + T^{\frac{1}{2}-b}\bigg\|\eta_{T}\frac{\pa_x}{\langle\epsilon\partial_x\rangle}\Big\{u^\pm+ e^{\pm \frac{2t}{\ep^2}\pa_x} u^\mp \Big\}^2\bigg\|_{X_{\pm,\ep}^{s,-(1-b)}}\\ 
   &\lesssim R+ T^{\frac{5}{4}-2b}\bigg\|\frac{\pa_x}{\langle\epsilon\partial_x\rangle}\Big\{u^\pm+ e^{\pm \frac{2t}{\ep^2}\pa_x} u^\mp \Big\}^2\bigg\|_{X_{\pm,\ep}^{s,-\frac{1}{4}}}
\end{aligned}
\end{equation}
Then, applying the bilinear estimates (Lemma \ref{Lem:BL1}, \ref{Lem:BL2} and \ref{Lem:BL3}) to the nonlinear term
$$\frac{\pa_x}{\langle\epsilon\partial_x\rangle}\Big\{u^\pm+ e^{\pm \frac{2t}{\ep^2}\pa_x} u^\mp \Big\}^2=\frac{\pa_x}{\langle\epsilon\partial_x\rangle}\Big\{(u^\pm)^2+ 2u^\pm\cdot e^{\pm \frac{2t}{\ep^2}\pa_x} u^\mp+e^{\pm \frac{2t}{\ep^2}\pa_x}(u^\mp)^2 \Big\}^2,$$
we obtain 
$$\|\mathbf{\Phi}_\epsilon(\mathbf{u})\|_{\mathbf{X}_{\ep}^{s,b}}\le cR + cT^{\frac{5}{4}-2b}\|\mathbf{u}\|_{\mathbf{X}_{\ep}^{s,b}}^2,$$
where $c>0$ is independent of $\ep\in(0,1]$. Similarly, for the difference, one can show that 
$$\|\mathbf{\Phi}_\epsilon(\mathbf{u}_1)-\mathbf{\Phi}_\epsilon(\mathbf{u}_2)\|_{\mathbf{X}_{\ep}^{s,b}}\le cT^{\frac{5}{4}-2b}\big(\|\mathbf{u}_1\|_{\mathbf{X}_{\ep}^{s,b}}+\|\mathbf{u}_2\|_{\mathbf{X}_{\ep}^{s,b}}\big)\|\mathbf{u}_1-\mathbf{u}_2\|_{\mathbf{X}_{\ep}^{s,b}}.$$
Therefore, taking $T=(8cR)^{-1/(\frac{5}{4}-2b)}$, we conclude that $\mathbf{\Phi}_\epsilon(\mathbf{u})$ is contractive on
$\{\mathbf{u}\in \mathbf{X}_{\ep}^{s,b} : \|\mathbf{u}\|_{\mathbf{X}_\ep^{s,b}} \le 2cR\}$, and thus the equation \eqref{Coupled Boussinesq} has a unique solution $u_\ep^{\pm}(t)$ with initial data $u_{\epsilon,0}^{\pm}$ such that $\|u_\ep^{\pm}\|_{X_{\ep,\pm}^{s,b}} \le 2cR$.
\end{proof}

\section{Proof of the main result}\label{sec: proof of the main result}

We are ready to prove the main theorem (Theorem \ref{main theorem'}), which is equivalent to Theorem \ref{main theorem}. For the proof, fix $0<s\leq 5$ and $b\in(\frac{1}{2}, \frac{5}{8})$ such that Proposition \ref{Pro : local sol to KdV} and \ref{Prop : local sol to Boussinesq} hold. Given $R>0$, we assume that $\sup_{\epsilon\in(0,1]}\|u_{\epsilon,0}^{\pm}\|_{H^s}\leq R$, and let $u_{\ep}^{\pm}(t)$ and $w_{\ep}^{\pm}(t)$ be the solutions to the coupled Boussinesq system \eqref{Boussinesq phase function} and the KdV equation \eqref{integral KdV}  with initial data $u_{\epsilon, 0}^\pm$, respectively, constructed in Proposition \ref{Prop : local sol to Boussinesq} and \ref{Pro : local sol to KdV}, and let $v_{\ep}^{\pm}(t)$ be the solution to the frequency localized decoupled Boussinesq equation \eqref{LFB} with initial data $P_{\leq N}u_{\epsilon, 0}^\pm$, given in Proposition \ref{Prop : local sol to Boussinesq}. Here, we may assume that $u_{\ep}^{\pm}(t)$, $v_{\ep}^{\pm}(t)$ and $w_{\ep}^{\pm}(t)$ all exist on the same time interval $[-T,T]$, where 
\begin{equation}\label{T choice}
T=\min\bigg\{\Big(\frac{c_0}{R}\Big)^{1/(\frac{5}{4}-2b)},1\bigg\}
\end{equation}
for some sufficiently small $c_0>0$ to be chosen later, and that they satisfy the uniform bounds, 
\begin{equation}\label{all solution uniform bounds}
\|u_{\ep}^{\pm}\|_{X_{\ep,\pm}^{0,b}}, \|v_{\ep}^{\pm}\|_{X_{\ep,\pm}^{0,b}}, \|w_{\ep}^{\pm}\|_{X_{\pm}^{0,b}}\lesssim R.
\end{equation}
Then, by the embedding for the Fourier restriction norm (Lemma \ref{Lem:Basic property of Xsb} $(3)$), it suffices to show that
$$\|u_{\ep}^{\pm}-v_{\ep}^{\pm}\|_{X_{\ep,\pm}^{0,b}}, \|v_{\ep}^{\pm}-w_{\ep}^{\pm}\|_{X_{\pm}^{0,b}}\lesssim_R \epsilon^{\min\{\frac{2s}{5},\frac{1}{2}\}}.$$

\subsection{Estimate for the difference $(u_{\ep}^{\pm}-v_{\ep}^{\pm})$}
Putting the temporal cut-off $\eta_T=\eta(\frac{\cdot}{T})$, where $\eta$ is a smooth bump function satisfying \eqref{bump function}, into \eqref{difference1}, we write the difference $u_{\ep}^{\pm}-v_{\ep}^{\pm}$ as 
$$\begin{aligned}
(u_{\ep}^{\pm}-v_{\ep}^{\pm})(t)&=\eta_T(t)S_\ep^{\pm}(t)(1-P_{\leq N})u_{\epsilon,0}^{\pm}\\
&\quad\pm \frac{\eta_T(t)}{2}\int_0^tS_\ep^{\pm}(t-t_1)\eta_{T}(t_1)\frac{\pa_x}{\langle\epsilon\partial_x\rangle}\big((u_\ep^\pm+v_\ep^\pm)(u_\ep^\pm-v_\ep^\pm)\big)(t_1)dt_1\\
&\quad\pm \frac{\eta_T(t)}{2}\int_0^tS_\ep^{\pm}(t-t_1)\eta_{T}(t_1)\frac{\pa_x}{\langle\epsilon\partial_x\rangle}\big( 2u_\ep^\pm e^{\pm \frac{2t_1}{\ep^2}\pa_x}u_\ep^\mp+ e^{\pm \frac{2t_1}{\ep^2}\pa_x} (u_\ep^\mp)^2\big) (t_1)dt_1.
\end{aligned}$$
Then, applying the linear estimates in the Fourier restriction norms (Lemma \ref{Lem:Basic property of Xsb}), we obtain  
$$\begin{aligned}
\|u_{\ep}^{\pm}-v_{\ep}^{\pm}\|_{X_{\epsilon, \pm}^{0,b}}&\lesssim T^{\frac{1}{2}-b}\|(1-P_{\leq N})u_{\epsilon,0}^{\pm}\|_{L^2}+ T^{\frac{1}{2}-b}\bigg\|\eta_{T}(t)\frac{\pa_x}{\langle\epsilon\partial_x\rangle}\big((u_\ep^\pm+v_\ep^\pm)(u_\ep^\pm-v_\ep^\pm)\big)\bigg\|_{X_{\epsilon, \pm}^{0,-(1-b)}}\\
&\quad + T^{\frac{1}{2}-b}\bigg\|\eta_{T}(t)\frac{\pa_x}{\langle\epsilon\partial_x\rangle}\Big( 2u_\ep^\pm e^{\pm \frac{2t_1}{\ep^2}\pa_x}u_\ep^\mp+ e^{\pm \frac{2t_1}{\ep^2}\pa_x} (u_\ep^\mp)^2\Big) \bigg\|_{X_{\epsilon, \pm}^{0,-(1-b)}}\\
&\lesssim T^{\frac{1}{2}-b}N^{-s}\|u_{\epsilon,0}^{\pm}\|_{H^s}+ T^{\frac{5}{4}-2b}\bigg\|\frac{\pa_x}{\langle\epsilon\partial_x\rangle}\big((u_\ep^\pm+v_\ep^\pm)(u_\ep^\pm-v_\ep^\pm)\big)\bigg\|_{X_{\epsilon, \pm}^{0,-\frac{1}{4}}}\\
&\quad + T^{\frac{5}{4}-2b}\bigg\|\frac{\pa_x}{\langle\epsilon\partial_x\rangle}\big(u_\ep^\pm e^{\pm \frac{2t_1}{\ep^2}\pa_x}u_\ep^\mp\big) \bigg\|_{X_{\epsilon, \pm}^{0,-\frac{1}{4}}}+ T^{\frac{5}{4}-2b}\bigg\|\frac{\pa_x}{\langle\epsilon\partial_x\rangle}e^{\pm \frac{2t_1}{\ep^2}\pa_x} \big((u_\ep^\mp)^2\big)\bigg\|_{X_{\epsilon, \pm}^{0,-\frac{1}{4}}}.
\end{aligned}$$
Hence, it follows from the bilinear estimates (Lemma \ref{Lem:BL1}, \ref{Lem:BL2} and \ref{Lem:BL3}) and the uniform bounds \eqref{all solution uniform bounds} with $N=\frac{1}{2}\epsilon^{-\frac{2}{5}}$ and the choice of  $T$ (see \eqref{T choice}) that 
$$\begin{aligned}
\|u_{\ep}^{\pm}-v_{\ep}^{\pm}\|_{X_{\epsilon, \pm}^{0,b}}&\lesssim N^{-s}T^{\frac{1}{2}-b}\|u_{\epsilon,0}^{\pm}\|_{H^s}+ T^{\frac{5}{4}-2b}\big(\|u_{\ep}^{\pm}\|_{X_{\epsilon, \pm}^{0,b}}+\|v_{\ep}^{\pm}\|_{X_{\epsilon, \pm}^{0,b}}\big)\|u_{\ep}^{\pm}-v_{\ep}^{\pm}\|_{X_{\epsilon, \pm}^{0,b}}\\
&\quad + \epsilon^{\min\{s,\frac{1}{2}\}}T^{\frac{5}{4}-2b}\|u_{\ep}^{\pm}\|_{X_{\epsilon, \pm}^{s,b}}\|u_{\ep}^{\mp}\|_{X_{\epsilon, \mp}^{s,b}}+ \epsilon^{\min\{s,1\}}T^{\frac{5}{4}-2b}\|u_{\ep}^{\pm}\|_{X_{\epsilon, \pm}^{s,b}}^2\\
&\lesssim \epsilon^{\frac{2s}{5}}T^{\frac{1}{2}-b}R+ \frac{c_0}{R}\cdot 2R\cdot\|u_{\ep}^{\pm}-v_{\ep}^{\pm}\|_{X_{\epsilon, \pm}^{0,b}}+\epsilon^{\min\{s,\frac{1}{2}\}}\cdot\frac{c_0}{R}\cdot R^2.
\end{aligned}$$
Since $c_0>0$ is small, this proves $\|u_{\ep}^{\pm}-v_{\ep}^{\pm}\|_{X_{\epsilon, \pm}^{0,b}}\lesssim_R \epsilon^{\min\{\frac{2s}{5},\frac{1}{2}\}}$.

\subsection{Estimate for the difference $(v_{\ep}^{\pm}-w_{\ep}^{\pm})$}
We consider the difference \eqref{difference2} of the form 
$$\begin{aligned}
(v_{\ep}^{\pm}-w_{\ep}^{\pm})(t)&=\eta_T(t)(1-P_{\leq N})w_{\ep}^{\pm}(t)+\eta_T(t)(S_\ep^{\pm}(t)-S^{\pm}(t))P_{\leq N}u_{\epsilon,0}^{\pm}\\
&\quad\pm \frac{\eta_T(t)}{2}\int_0^tS_\ep^{\pm}(t-t_1)\eta_{T}(t_1)\frac{\pa_x}{\langle\epsilon\partial_x\rangle}P_{\leq N}\big((v_\ep^\pm+w_\ep^\pm)(v_\ep^\pm-w_\ep^\pm)(t_1)\big)dt_1\\
&\quad\pm \frac{\eta_T(t)}{2}\int_0^t(S_\ep^{\pm}(t-t_1)-S^{\pm}(t-t_1))\eta_{T}(t_1)\frac{\pa_x}{\langle\epsilon\partial_x\rangle}P_{\leq N}(w_\epsilon^{\pm}(t_1))^2dt_1\\
&\quad\mp \frac{\eta_T(t)}{2}\int_0^tS^{\pm}(t-t_1)\eta_{T}(t_1)\partial_x\bigg( 1 - \frac{1 }{\langle\epsilon\partial_x\rangle} \bigg)P_{\leq N}(w_\epsilon^{\pm}(t_1))^2dt_1.
\end{aligned}$$
Then, as before, applying the linear estimates (Lemma \ref{Lem:difference linear sol in Xsb} to the second and the fourth terms containing the difference of two propagators, and Lemma \ref{Lem:Basic property of Xsb} to the other terms), we obtain 
$$\begin{aligned}
\|v_{\ep}^{\pm}-w_{\ep}^{\pm}\|_{X_{\pm}^{0,b}}&\lesssim N^{-s}\|w_{\ep}^{\pm}\|_{X_{\pm}^{s,b}}+\ep^{\frac{2s}{5}}T^{\frac{3}{2}-b}\|u_{\epsilon,0}^{\pm}\|_{H^s}+T^{\frac{5}{4}-2b}\big\|\pa_x\big((v_\ep^\pm+w_\ep^\pm)(v_\ep^\pm-w_\ep^\pm)\big)\big\|_{X_{\pm}^{0,-\frac{1}{4}}}\\
&+\ep^{\frac{2s}{5}}T^{\frac{9}{4}-b}\|\pa_x(w_\epsilon^{\pm})^2\|_{X_{\pm}^{s,-\frac{1}{4}}}+\epsilon^{\frac25s} T^{\frac{5}{4}-2b}\|\pa_x(w_\epsilon^{\pm})^2\|_{X_{\pm}^{s,-\frac{1}{4}}},
\end{aligned}$$
where for the last term, we used that $1-\frac{1}{\langle\epsilon\xi\rangle}=\frac{\epsilon^2\xi^2}{\langle\epsilon\xi\rangle(1+\langle\epsilon\xi\rangle)}\lesssim \epsilon^{\frac25s}|\xi|^{\frac25s}\lesssim \epsilon^{\frac25s}\langle \xi\rangle^{s}$ when $|\xi|\leq \frac12\ep^{-\frac25}$. Then, the bilinear estimate \eqref{bilinear KdV} for the Airy flows yields 
$$\begin{aligned}
\|v_{\ep}^{\pm}-w_{\ep}^{\pm}\|_{X_{\pm}^{0,b}}&\lesssim \epsilon^{\frac{2s}{5}}\|w_{\ep}^{\pm}\|_{X_{\pm}^{s,b}}+\ep^{\frac{2s}{5}}T^{\frac{3}{2}-b}\|u_{\epsilon,0}^{\pm}\|_{H^s}+T^{\frac{5}{4}-2b}\big(\|v_{\ep}^{\pm}\|_{X_{\pm}^{0,b}}+\|w_{\ep}^{\pm}\|_{X_{\pm}^{0,b}}\big)\|v_{\ep}^{\pm}-w_{\ep}^{\pm}\|_{X_{\pm}^{0,b}}\\
&+\ep^{\frac{2s}{5}}T^{\frac{9}{4}-b}\|w_{\ep}^{\pm}\|_{X_{\pm}^{s,b}}^2+\epsilon^{\frac25s} T^{\frac{5}{4}-2b}\|w_{\ep}^{\pm}\|_{X_{\pm}^{s,b}}^2.
\end{aligned}$$
Note that by construction, the Fourier transform of $v_\epsilon^\pm$ is supported in $|\xi|\leq N$. Hence, by Lemma \ref{Lem:frequency loc norm}, $\|v_{\ep}^{\pm}\|_{X_{\pm}^{s,b}}\sim \|v_{\ep}^{\pm}\|_{X_{\epsilon, \pm}^{s,b}}\lesssim R$. Thus, it follows from the uniform bounds \eqref{all solution uniform bounds} that 
$$\|v_{\ep}^{\pm}-w_{\ep}^{\pm}\|_{X_{\pm}^{0,b}}\lesssim \epsilon^{\frac{2s}{5}}R+\ep^{\frac{2s}{5}}T^{\frac{3}{2}-b}R+T^{\frac{5}{4}-2b}R\|v_{\ep}^{\pm}-w_{\ep}^{\pm}\|_{X_{\pm}^{0,b}}+\ep^{\frac{2s}{5}}T^{\frac{9}{4}-b}R^2+\epsilon^{\frac25s} T^{\frac{5}{4}-2b}R^2.$$
Therefore, by the choice of $T$ in \eqref{T choice}, we conclude that  $\|v_{\ep}^{\pm}-w_{\ep}^{\pm}\|_{X_{\pm}^{0,b}}\lesssim_R \epsilon^{\frac{2s}{5}}$.

\end{document}